\documentclass[11pt]{amsart}
\usepackage{amsfonts,amssymb,amscd,amsmath,enumerate,verbatim,calc}

\def\NZQ{\mathbb}               % the font for N,Z,Q,R,C
\def\NN{{\NZQ N}}

\def\ZZ{{\NZQ Z}}

\def\frk{\mathfrak}               % font for "Fraktur"

\def\Phi{{\frk N}}

\def\opn#1#2{\def#1{\operatorname{#2}}} % to make operators
\opn\chara{char} \opn\length{\ell} \opn\pd{pd} \opn\rk{rk}
\opn\projdim{proj\,dim} \opn\injdim{inj\,dim} \opn\rank{rank}
\opn\depth{depth} \opn\grade{grade}
\opn\height{height}\opn\sdepth{sdepth} \opn\divides{divides} \opn\embdim{emb\,dim}
\opn\codim{codim}

\opn\Tr{Tr} \opn\bigrank{big\,rank}
\opn\superheight{superheight}\opn\lcm{lcm} \opn\gcd{gcd}
\opn\trdeg{tr\,deg}%\emph{
\opn\reg{reg} \opn\lreg{lreg} \opn\ini{in} \opn\lpd{lpd}
\opn\size{size}\opn{\mult}{mult}\opn\Mon{Mon} \opn\lex{lex}
\opn\div{div} \opn\Div{Div} \opn\cl{cl} \opn\Cl{Cl}
\opn\Spec{Spec} \opn\Supp{Supp} \opn\supp{supp} \opn\Sing{Sing}
\opn\Ass{Ass} \opn\min{min} \opn\max{max}
\opn\Ann{Ann} \opn\Rad{Rad} \opn\Soc{Soc}
\opn\Syz{Syz} \opn\Im{Im} \opn\Ker{Ker} \opn\Coker{Coker}
\opn\Am{Am} \opn\Hom{Hom} \opn\Tor{Tor} \opn\Ext{Ext}
\opn\End{End} \opn\Aut{Aut} \opn\id{id}

\opn\nat{nat}
\opn\pff{pf}%   \pf exists already
\opn\Pf{Pf} \opn\GL{GL} \opn\SL{SL} \opn\mod{mod} \opn\ord{ord}
\opn\Gin{Gin}
\opn\Hilb{Hilb}\opn\adeg{adeg}\opn\std{std}\opn\ip{infpt}
\opn\Pol{Pol} \opn\sat{sat} \opn\Var{Var}

\opn\aff{aff} \opn\con{conv} \opn\relint{relint} \opn\st{st}
\opn\lk{lk} \opn\cn{cn} \opn\core{core} \opn\vol{vol}
\opn\link{link} \opn\star{star}
\opn\gr{gr}

\def\Dc{{\mathcal D}}
\def\polykxn{K[x_1,\dots,x_{n}]}
\def\pot#1#2{#1[\kern-0.28ex[#2]\kern-0.28ex]}

\opn\dirlim{\underrightarrow{\lim}}
\opn\inivlim{\underleftarrow{\lim}}
\let\iso=\cong

\def\Implies{\ifmmode\Longrightarrow \else
        \unskip${}\Longrightarrow{}$\ignorespaces\fi}
\def\implies{\ifmmode\Rightarrow \else
        \unskip${}\Rightarrow{}$\ignorespaces\fi}
\def\iff{\ifmmode\Longleftrightarrow \else
        \unskip${}\Longleftrightarrow{}$\ignorespaces\fi}

\let\:=\colon
\newtheorem{Theorem}{Theorem}[section]
\newtheorem{Lemma}[Theorem]{Lemma}
\newtheorem{Corollary}[Theorem]{Corollary}
\newtheorem{Proposition}[Theorem]{Proposition}

\let\epsilon\varepsilon
\let\phi=\varphi
\let\kappa=\varkappa
\def\qed{\ifhmode\textqed\fi
      \ifmmode\ifinner\quad\qedsymbol\else\dispqed\fi\fi}
\def\textqed{\unskip\nobreak\penalty50
       \hskip2em\hbox{}\nobreak\hfil\qedsymbol
       \parfillskip=0pt \finalhyphendemerits=0}
\def\dispqed{\rlap{\qquad\qedsymbol}}

\opn\dis{dis}
\def\pnt{{\raise0.5mm\hbox{\large\bf.}}}

\opn\Lex{Lex}

\begin{document}
\title{Stanley's  conjecture for  critical ideals}
\author{Azeem Haider and Sardar Mohib Ali Khan}
\keywords{Critical monomial ideals, Stanley depth, Hilbert
function\\ \indent This research of the authors is partially supported by HEC Pakistan.}
\address{Abdus Salam School of Mathematical Sciences, GC University, 68-B, New Muslim Town, Lahore, Pakistan.}
\email{azeemhaider@gmail.com {\rm and }ali.uno@gmail.com}
\date{}
\maketitle
\begin{abstract}
Let $ S=\polykxn$ be a polynomial ring in $n$ variables over a
field $K$. Stanley's conjecture holds for the modules $I$ and $S/I,$
when $I \subset S$ is a critical monomial ideal. We calculate the
Stanley depth of $S/I$ when $I$ is a canonical critical monomial
ideal.  For non critical monomial ideals we show the existence of
a Stanley ideal with the same depth and Hilbert function.
\end{abstract}
\section{Introduction}
Let $ S=\polykxn$ be a polynomial ring with standard
grading over a field $K$. Let $M$ be a finitely generated
$\ZZ^n$-graded $S$-module. Any decomposition of the module $M$ as a finite
direct sum of $\ZZ^{n}$-graded $K$-subspaces of the form $uK[Z]$,
where each $uK[Z]$ is a free  $K[Z]$-modules, is
called a {\it{Stanley decomposition}} of the module $M$. In other words, a Stanley decomposition of $M$ has the form  $\Dc
:M=\bigoplus_{i=1}^{r}u_{i}K[Z_{i}]$, where the $u_i\in M$ are
homogenous elements and  each $Z_i$ is a subset of $\{x_{1},\ldots,x_{n}\}$.  The number
$ \sdepth (\Dc)=\min\{ |Z_i|,\, i=1,2,\ldots,r\}$ is called the
{\it{Stanley depth}} of the decomposition $\Dc.$ The Stanley depth
of the module $M$ is defined to be $\sdepth(M)=\max\{
\sdepth(\Dc): \Dc \text{ is Stanley decomposition} \}.$

In \cite{St} Stanley conjectured that $\sdepth(M) \geq \depth(M)$.
We call $I$ a {\it{Stanley ideal}}, if Stanley's conjecture holds
for $S/I$. There are not many known classes of Stanley
ideals \cite{DP}.

Let $I \subset S$ be a monomial ideal. We denote by $G(I)$
the unique minimal monomial system of generators of $I$ and
$H_{S/I}$ the Hilbert function of the quotient algebra $S/I$.
Consider the lexicographic order $<_{\lex}$ on $S$ induced by the
ordering $x_{1}>x_{2}>\ldots>x_{n}$ of the variables. A
{\it{lexsegment ideal}} is a monomial ideal $I$ such that for a
monomial $u \in I$ and for a monomial $v \in S$ with $ \deg \, u =
\deg \,v$ and $ v
>_{\lex} u$, one has $v \in I$. A lexsegment ideal $I$ is called a
{\it{universal lexsegment ideal}} if $I$ is a lexsegment ideal in
$K[x_{1},\ldots,x_{n+m}]$ for any natural
number $m \geq 0$.

Recall that for any graded ideal $I \subset S$, there
exists a unique lexsegment ideal, denoted by $I^{\lex}$, such that
$S/I$ and $S/I^{\lex}$ have the same Hilbert function. Hibi and
Murai \cite{HA} call a monomial ideal  $I$ {\it{critical}} if
$I^{\lex}$ is
universal lexsegment.

Let $m_{1},m_{2},\ldots,m_{t} \in \Mon(S)$ for $1 \leq t
\leq n$ where  $m_{i} \in K[x_{i},x_{i+1},\ldots,x_{n}]$ and $\deg m_t>0$.  Then we define
the ideal
\begin{equation}\label{04}
I_{(m_{1},m_{2},\ldots,m_{t})}=(x_{1}m_{1}, x_{2}m_{1}m_{2},\ldots, x_{t-1}m_{1}m_{2}\cdots
m_{t-1}, m_{1}m_{2}\cdots m_{t})
\end{equation}

In analogy to the definition of Hibi and Murai \cite{HB} we call a
monomial ideal {\it {canonical critical}}, if it is of the form
$I_{(m_{1},m_{2},\ldots,m_{t})}$, up to the permutation of
variables.  By \cite[Theorem 1.1]{HB} canonical critical ideals  are critical. In
Lemma \ref{p1} we show that Stanley's conjecture holds for $I$ and
$S/I$. For a canonical critical monomial ideal $I$ we calculate the
Stanley depth of $S/I$ (Theorem \ref{t1}) and obtain a Stanley
decomposition (Theorem \ref{t2}) which exactly gives the Stanley
depth of $S/I$. We also show that for a canonical critical
monomial ideal one has  $\sdepth(I) \geq 1+\sdepth(S/I) $, thereby
giving in this special case an affirmative answer to a question
raised by Rauf in \cite{AR}. In Proposition \ref{t3} we show that for each non critical monomial ideal $I$ there
exists  a Stanley ideal  which has
the same depth and Hilbert function as the ideal $I$.\\
\indent {\bf Acknowledgements.} The authors like to thank Professor J\"{u}rgen Herzog for his valuable suggestions which improves the final form of the paper.
\section{Stanley depth and Critical Monomial Ideals}
First we show that the Stanley conjecture holds for
modules $I$ and $S/I$, when $I$ is critical monomial ideal.
\begin{Lemma}\label{p1} Let $I \subset S $ be a critical monomial
ideal. Then
\begin{enumerate}
\item[{\rm (i)}] $\sdepth(S/I) \geq \depth(S/I);$
\item[{\rm(ii)}]$\sdepth(I) \geq  \depth(I).$
\end{enumerate}
\end{Lemma}
\begin{proof}
(i) If $I\subset S$ is a critical monomial ideal, then
$\depth(S/I)=n-|G(I)|$ (see \cite[Theorem 1.6]{HA}) and for any
monomial ideal $I$,  $\sdepth(S/I) \geq n-|G(I)|$ (see
\cite[Proposition~1.3]{MC}).

(ii) follows from the fact
$\depth(I)=1+\depth(S/I)$ and $\sdepth(I) \geq \max \{1,n-|G(I)|+1
\} $ (see \cite[Proposition 3.4]{HM}).
\end{proof}

We show that the equality holds in (i) of  Lemma \ref{p1}  for
canonical critical monomial ideals.
\begin{Theorem}\label{t1}
Let  $I \subset S$ be a canonical critical monomial ideal. Then
$\sdepth(S/I)=\depth(S/I)=n-|G(I)|$.
\end{Theorem}
\begin{proof} If $I$ be a canonical critical monomial ideal, then
$$I=(x_{1}m_{1},x_{2}m_{1}m_{2},\ldots,x_{t-1}m_{1}m_{2}\cdots
m_{t-1},m_{1}m_{2}\cdots m_{t}),$$
where for $1\leq t\leq n$,  $m_{i}$ is a monomial belonging to $K[x_{i},x_{i+1},\ldots,x_{n}]$ and $\deg m_t>0$.

We set  $S_i=K[x_i,x_{i+1},\ldots,x_n]$ and for $1\leq i\leq t-1$ we define the ideals
$$I_i=(x_{i}m_{i},x_{i+1}m_{i}m_{i+1},\ldots,x_{t-1}m_{i}m_{i+1}\cdots
m_{t-1},m_{i}m_{i+1}\cdots m_{t}),$$
and
$$I_i'=(x_{i},x_{i+1}m_{i+1},\ldots,x_{t-1}m_{i+1}m_{i+2}\ldots
m_{t-1},m_{i+1}m_{i+2}\ldots m_{t}).$$
Moreover, we set $I_t'=(m_t)$.

Then $S=S_1$, $I=I_1$, $I_i=m_iI_i'$ and
$I_i'=(x_i,I_{i+1})$. The ideals $I_i$, $I_i'$
and $I_t'$ are critical monomial ideals in
$S_j=[x_j,x_{j+1},\ldots,x_n]$ for each $1\leq j\leq i$.

By \cite[Proposition 1.3]{MC} and the fact that $\gcd\{u|u\in G(I)\}=m_1$, we have  $\sdepth(S/I)=\sdepth(S/{I_1'})$,  and further   $\sdepth(S/{I_1'})=\sdepth(S_2/I_2)$, since $S/{I_1'}\iso S_2/I_2$. Hence we conclude that $\sdepth(S/I)=\sdepth(S_2/I_2)$.
Continuing in this way we get $\sdepth(S/I)=\sdepth(S_t/{I_t'})=n-|G(I)|$. The last equation follows since  $I_t'=(m_t)$ is a principal monomial ideal, and since $t=|G(I)|$ (see \cite[Corollary 2.3]{HB}).
\end{proof}

\begin{Corollary}\label{c1}
Let $I$ is a canonical critical monomial ideal.  Then
\[
\sdepth(I)\geq 1+\sdepth(S/I).
\]
\end{Corollary}

\begin{proof}
If $I$ is a canonical critical monomial ideal, then by using
 \cite[Proposition 3.4]{HM}  we have $\sdepth(I)\geq n-|G(I)|+1$,  and the assertion follows from
  Theorem \ref{t1}.
  \end{proof}

For a canonical critical monomial ideal $I$, the following
decomposition is obtained
 from \cite[Lemma 2.2]{HB}.

\begin{Lemma}\label{l7}
As a vector space over $K$ the canonical critical monomial ideal
$I$ is the direct sum
$I=\bigoplus\limits_{j=1}^{t-1}x_j(\prod\limits_{k=1}^{j}m_k)
K[x_j,x_{j+1},\ldots,x_n]\bigoplus
\prod\limits_{k=1}^{t}m_kK[x_t,\ldots,x_n]$.
\end{Lemma}

Using the decomposition in Lemma \ref{l7} we will give an explicit Stanley
decomposition of $S/I$ which in fact gives the Stanley depth.

Let $S_i=K[x_i,\ldots,x_n]$ and $m_i \in S_i$ with $\deg m_i = d_i$ for $ 1\leq i \leq t \leq n.$ We define for any monomial  $m\in S$  a positive number
$v (m)=\min \{k\,|\; x_k\, \divides \, m \}$ and monomials $w_{i(d+1)}= \frac{w_{id}}{x_{v(w_{id})}},$ where $ w_{i1} = m_i,$ $ 1 \leq d \leq d_i -1$ and $1 \leq i \leq t. $
For monomial $m_i$ we have $m_i = \prod\limits_{j=1}^{d_i} x_{v(w_{ij})}.$
Finally, by using the monomials $w_{id}$,  we define $u_{ij} = \prod\limits_{k=1}^{j-1} x_{v(w_{ik})}$
with $u_{i1}= 1$ and $Z_{ij}= \{ x_i,x_{i+1},\ldots,x_n \}\setminus \{ x_{v(w_{ij})}\}$.

With the notation introduced we have,
\begin{Theorem}\label{t2}
Let  $I=I_{(m_{1},m_{2},\ldots,m_{t})}$  be a critical monomial ideal. We set $n_1=1$ and $n_i=m_1m_2\cdots m_{i-1}$ for $i=2,\ldots,t$.  Then for $S/I$ we have the following Stanley decomposition
\[{\mathcal D}: S/I=\bigoplus\limits_{i=1}^{t}\bigoplus\limits_{j=1}^{d_i}u_{ij}
n_iK[Z_{ij}] \quad \text{with}\quad
d_i=\deg(m_i).
\]
Moreover we have
$\sdepth({\mathcal D})=\sdepth(S/I).$
\end{Theorem}

\begin{proof}
We decompose $S_i$ for the variable $x_{v(w_{i1})} = u_{i2}$,
$$S_i = K[Z_{i1}]\oplus x_{v(w_{i1})}S_i. $$
Again for $x_{v(w_{i2})}$ we decompose $S_i$ in the above equation,
$$S_i =K[Z_{i1}]\oplus x_{v(w_{i1})}K[Z_{i2}]\oplus x_{v(w_{i1})}x_{v(w_{i2})}S_{i}$$
$$=u_{i1}K[Z_{i1}] \oplus u_{i2}K[Z_{i2}] \oplus u_{i3}S_{i}.$$
We know that $m_i = \prod \limits_{j=1}^{d_i} x _{v(w_{ij})}$, so continuing in this way we obtain
\begin{equation}\label{eq1}
S_i = \bigoplus\limits_{j=1}^{d_i} u_{ij} k[Z_{ij}] \oplus m_i S_i.
\end{equation}
As $S_i=S_{i+1}\oplus x_{i}S_{i}$ for  $1 \leq i \leq t-1,$ it follows from   (\ref{eq1}) that
\begin{equation}\label{eq2}
S_i= \bigoplus\limits_{j=1}^{d_i} u_{ij} k[Z_{ij}] \oplus m_i S_{i+1} \oplus x_{i}m_{i}S_{i}.
\end{equation}
By using the above recursive relation for $1 \leq i \leq t-1$ we get
\begin{equation}\label{eq3}
S_1 = \bigoplus\limits_{i=1}^{t-1}\bigoplus\limits_{j=1}^{d_i} u_{ij} n_i K[Z_{ij}]\oplus \bigoplus\limits_{j=1}^{t-1}x_j n_{j+1} S_j \oplus n_{t}S_{t}.
\end{equation}
Now for $S_t$ we substitute the decomposition obtained in (\ref{eq1})
$$S_1 = \bigoplus\limits_{i=1}^{t-1}\bigoplus\limits_{j=1}^{d_i} u_{ij} n_i K[Z_{ij}]\oplus \bigoplus\limits_{j=1}^{t-1}x_j n_{j+1} S_j \oplus n_{t}\left( \bigoplus\limits_{j=1}^{d_t}u_{tj}K[Z_{tj}]\oplus m_t S_t\right)$$
and obtain
\begin{equation}\label{02}
S_1 = \bigoplus\limits_{i=1}^{t}\bigoplus\limits_{j=1}^{d_i} u_{ij} n_i K[Z_{ij}]\oplus \bigoplus\limits_{j=1}^{t-1}x_j n_{j+1} S_j \oplus n_{t}m_{t} S_t.
\end{equation}

Since $I$ is a critical monomial ideal, Lemma \ref{l7} implies that
$I=\bigoplus\limits_{j=1}^{t-1}x_j n_{j+1}
S_j\oplus
n_t m_t S_t$ and that $S_1=S$.  Thus  (\ref{02}) yields
\begin{equation}\label{03}
S/I= \bigoplus\limits_{i=1}^{t}\bigoplus\limits_{j=1}^{d_i} u_{ij} n_i K[Z_{ij}],
\end{equation}
where  for each $i$ and $j$ we have $|Z_{ij}|=n-i\geq n-t$. This implies  $|Z_{ij}|\geq
n-|G(I)|$ as $t=|G(I)|$. Hence the desired conclusion follows from Theorem \ref{t1}.
\end{proof}
\indent  Recall that
a numerical function $H\:\NN\rightarrow \NN$ is the  Hilbert function of $S/I$
for some graded ideal  $I\subset
S=\polykxn$,  if and only if $H(0)=1,$ $H(1)\leq n$ and $H(d+1)\leq H(d)^{<d>}$ for all
   $d\geq 1$ (See \cite[Theorem 4.2.10]{BH}).\\

\indent If $I$ is a non critical monomial ideal we prove the
following.
\begin{Proposition}\label{t3}
Let $K$ be an infinite field. Then for any non critical monomial ideal $I\subset S=\polykxn$ there exists a Stanley  ideal $L\subset
S$ such that $S/I$ and $S/L$ have the same depths and  the same Hilbert
function.
\end{Proposition}
\begin{proof} Let $I \subset S$ be a non critical monomial ideal and
$I^{\lex}\subset S$ is the corresponding lexsegment ideal then by
\cite[Corollary 1.3]{HA} we have $|G(I^{\lex})|> n$. If
$\depth(S/I)=b$, then there exists a regular sequence
$(\theta_1,\theta_2,\ldots, \theta_b)$ of $S/I$ with each
$\deg(\theta_i)=1$. It then follows that there exists a
homogeneous ideal $J$ of $S'=K[x_1,\ldots, x_{n-b}]$ such that
the ideal $JS$ of $S$ satisfies $H_{S/JS}=H_{S/I}$.

We now claim
that the lexsegment ideal $J^{\lex}\subset S'$ of $J$ cannot be
universal lexsegment. In fact, if $J^{\lex}$ is universal
lexsegment, then $J^{\lex}$ remains being lexsegment in the
polynomial ring $K[x_1,\ldots, x_m]$ for each $m\geq n-b$. In
particular, the ideal $J^{\lex}S$ of $S$ is universal lexsegment.
Since $H_{S/JS}= H_{S/J^{\lex}S}=H_{S/I}$, it follows that
$I^{\lex}=J^{\lex}S$, because we have unique lexsegment ideal
corresponding to every ideal $I \subset S$. Thus $I$ has a universal
lexsegment ideal $I^{\lex},$ which is a contradiction.

Since the
lexsegment ideal $J^{\lex}$ of $J$ cannot be universal lexsegment,
it follows from \cite[Corollary 1.3,1.4]{HA} that
$\depth(S'/J^{\lex})=0$. Thus $\depth(S/J^{\lex}S)=b \leq
\sdepth_{S'} (S'/J^{\lex}) + b = \sdepth_{S}(S/J^{\lex}S), $
by (\cite[Lemma 3.6]{HM}).
\end{proof}

\end{document}